\documentclass[11pt]{amsart}
\usepackage[T1]{fontenc}
\usepackage[english]{babel}
\usepackage{libertine}

\usepackage{amsmath,amssymb,amsthm, enumitem,xspace}
\usepackage{comment}
\usepackage{accents}
\usepackage[colorlinks=true,citecolor=blue,linkcolor=blue]{hyperref}
\usepackage{cleveref}
\usepackage{tikz}
\usetikzlibrary{arrows.meta} 
\usepackage{mathrsfs}

\newtheorem{thm}{Theorem} 
\newtheorem*{thm*}{Theorem}

\newtheorem{prop}[thm]{Proposition}

\newtheorem{ithm}{Theorem}

\newtheorem*{iprob*}{Problem}

\theoremstyle{definition}

\newtheorem{rem}[thm]{Remark}

\newcommand*{\dd}{\mathrm{d}}

\newcommand*{\ZZ}{\mathbf{Z}}

\newcommand*{\RR}{\mathbf{R}}
\newcommand*{\NN}{\mathbf{N}}

\newcommand*{\CC}{\mathbf{C}}
\newcommand*{\DD}{\mathbf{D}}

\newcommand*{\SL}{\mathbf{SL}}

\newcommand*{\Aut}{{\rm Aut}}

\newcommand*{\se}{\subseteq}
\newcommand*{\inv}{^{-1}}

\newcommand*{\fhi}{\varphi}

\newcommand*{\smooth}[1]{{#1}_{\mathrm{sm}}}

\newcommand*{\one}{\boldsymbol{1}}

\DeclareMathOperator{\Fix}{Fix}
\DeclareMathOperator{\diam}{diam}
\title[Inadmissible representations of the tree automorphism group]{Inadmissible representations\\ of the tree automorphism group}

\author[Nicolas Monod]{Nicolas Monod}
\address{
EPFL, Switzerland}
\begin{document}

\begin{abstract}
The automorphism group of a regular locally finite tree is shown to admit irreducible Banach representations that are not admissible. The dense subspace of smooth vectors contains no algebraically irreducible component.
\end{abstract}
\maketitle


\section{Introduction}
\subsection{Context}
Classically, a linear representation of a group $G$ on a Banach space $V$ is called irreducible if there is no $G$-invariant closed proper subspace ${V_0\se V}$; proper means $V_0\neq0,V$.

In this note, we call this \textbf{topologically irreducible} to distinguish it from the following. The representation is \textbf{algebraically irreducible} if there is no $G$-invariant proper subspace ${V_0\se V}$ at all, with $V$ viewed merely as a vector space.

A straightforward Baire argument shows that algebraic irreducibility can never happen when $V$ is an infinite-dimensional Banach space and $G$ is countable, or more generally totally disconnected locally compact second countable and the representation is continuous.

That setting includes notably $p$-adic algebraic groups.

Nevertheless, algebraic irreducibility plays a crucial role in the very rich representation theory of such groups because one can replace the Banach space $V$ with the subspace of \textbf{smooth} vectors $\smooth{V} \se V$, namely those with an open stabilizer (which in this setting coincides with the ``$K$-finite'' vectors). Whenever algebraically irreducible smooth representations can be classified, this allows \emph{some} classification results, most notably when $V$ is a Hilbert space and the representation is unitary, see e.g.~\cite[\S 4.21]{Bernstein-Zelevinskii76} and~\cite[p.~30]{Jacquet-Langlands}.


\medskip
In this context, the key concept is \textbf{admissibility}, namely the requirement that open subgroups $U<G$ have only a finite-dimensional space $V^U$ of fixed vectors. Admissibility is the gateway to classification, when classification is possible. A cornerstone of the study of representations of reductive $p$-adic groups is to prove admissibility of irreducible representations, see e.g.~\cite[VI.2.2]{Renard_book}.

\subsection{Tree groups}
Early on, it was recognized that the automorphism group $G=\Aut(T)$ of a regular tree $T$ of finite valency~$\geq 3$ offers a very interesting non-linear analogue to $p$-adic groups. One of the reasons is that $T$ occurs as the Bruhat--Tits building~\cite{Bruhat-Tits1967,Bruhat-Tits1972} of rank one $p$-adic groups when the valency is $p+1$; and yet $\Aut(T)$ is much larger than an algebraic group. The representation theory of $\Aut(T)$ was initiated by Cartier~\cite{Cartier_bourbaki,Cartier73} and the unitary case studied in great detail in~\cite{Figa-Talamanca-Nebbia}. Indeed the unitary dual is completely known; this follows from a complete description by Olshanskii of all algebraically irreducible smooth representations~\cite{Olshanskii77_en}. Again, this depends crucially on the fact that they are all admissible, which was proved earlier by Olshanskii~\cite{Olshanskii75}.

The situation is very different for topological irreducibility:

\begin{ithm}\label{ithm:inad}
Let $G$ be the automorphism group of a regular tree of finite valency~${\geq 3}$.

There exists topologically irreducible continuous $G$-representations on a Banach space $V$ that are not admissible.

Moreover, the subspace $\smooth V$ of smooth vector does not contain any non-zero algebraically irreducible subrepresentation.
\end{ithm}

In particular, the classification of algebraically irreducible smooth representations gives no information at all about $V$. By contrast, all topologically irreducible unitary $G$-representations are admissible, see e.g.~\cite{Figa-Talamanca-Nebbia}. Very rencently, the corresponding statement has been established by Viola~\cite{Viola_tree_arx} in the more general setting of Hilbert space representations preserving a sesquilinear form of any finite index.


\medskip
Two remarks are in order.

The first is that \Cref{ithm:inad} is a non-Archimedean analogue of Soergel's result~\cite{Soergel} that $\SL(2,\RR)$ admits inadmissible topologically irreducible continuous representations. Although the techniques are different because $G=\Aut(T)$ does not give rise to Harish-Chandra modules, the strategy is parallel in that we ``induce'' from counter-examples to the invariant subspace problem.

The second remark is, therefore, that Soergel's own two remarks apply: the hard work is done by Enflo and Read to provide these counter-examples, and the idea to induce them is attributed to Vogan.

\medskip
Given the connection to the invariant subspace problem, it is of keen interest to investigate which Banach spaces $V$ can occur in \Cref{ithm:inad}. Leveraging the work of Read, we obtain:

\begin{ithm}\label{ithm:more-V}
In \Cref{ithm:inad}, one can arrange that $V$ is isometrically isomorphic to its own bidual. Alternatively, one can choose $V=L^1$.
\end{ithm}

The usual convention is that $L^1$ denotes a classical Lebesgue space, for instance $L^1([0,1])$ or $L^1(\RR)$, which are isomorphic as Banach spaces.

\section{Proofs}
\subsection{The tree}
We fix throughout an integer $q\geq 2$ and consider the $(q+1)$-regular tree $T$ and its automorphism group $G=\Aut(T)$ endowed with the locally compact topology of pointwise convergence. We identify $T$ with its vertex set endowed with the graph distance $d$.

A \textbf{geodesic ray} in $T$ is a sequence $(x_k)_{k\in\NN}$ of vertices such that $d(x_k, x_n) = |k-n|$ for all $k,n\in \NN$; we say that this ray is \textbf{issuing} from $x_0$. Geodesic lines are defined similarly but with $k\in \ZZ$. The \textbf{ideal boundary} $\partial T$ of $T$ refers to the set of equivalence classes of all geodesic rays, where $(x_k)$ and $(y_k)$ are deemed equivalent if there is $m\in \ZZ$ such that $x_k = y_{k+m}$ for all $k$ large enough.

An equivalent definition of $\partial T$ is as follows: fix some $x_0$ and let $\partial T$ be the set of all geodesic rays issuing from $x_0$. The advantage of this second model is that it gives an identification of $\partial T$ with a product of finite sets,
\begin{equation*}
\partial T \cong \{1, \ldots, q+1\} \times \{1, \ldots, q\}^\NN,
\end{equation*}
obtained by recording each successive choice for $x_{k+1}$ among the neighbours of $x_k$. The geodesic condition excludes precisely one of $q+1$ choices when $k\geq 1$, namely $x_{k-1}$. This identification endows $\partial T$ with a compact topology and a probability measure $\mu_{x_0}$, namely the product of all corresponding uniform measures on $\{1, \ldots, q+1\}$, respectively $\{1, \ldots, q\}$. 

An equivalent definition of the topology and the measure comes from the observation that the compact (profinite) group $\Fix_G(x_0)$ acts transitively on $\partial T$; in particular the normalized Haar measure of $\Fix_G(x_0)$ induces the unique $\Fix_G(x_0)$-invariant probability measure of $\partial T$.

The drawback of these descriptions is that they depend on the choice $x_0$. It is however apparent that another choice of basepoint, say $y_0$, will define the same topology and the same measure \emph{class}. Moreover, the Radon--Nikod{\'y}m derivative of the measures can readily be computed as
\begin{equation*}
\frac{\dd \mu_{y_0}}{\dd \mu_{x_0}} (\xi) = q^{B_\xi(x_0, y_0)}\kern3mm \forall \, \xi\in\partial T,
\end{equation*}
where $B_\xi$ is the \textbf{Busemann kernel} defined by
\begin{equation}\label{eq:Busemann}
B_\xi(x_0, y_0) = \lim_{k\to\infty} \Big( d(x_0, z_k) - d(y_0, z_k)\Big) \in \ZZ
\end{equation}
with $(z_k)$ any geodesic ray representing $\xi$. In particular, the $G$-action on $(\partial T, \mu_{x_0})$ is non-singular, with Radon--Nikod{\'y}m cocycle
\begin{equation*}
\frac{\dd g \mu_{x_0}}{\dd \mu_{x_0}} (\xi) = q^{B_\xi(x_0, g x_0)}\kern3mm \text{for} \kern3mm  g\in G, \xi\in\partial T.
\end{equation*}

\begin{figure}[hbt]
\centering
\input{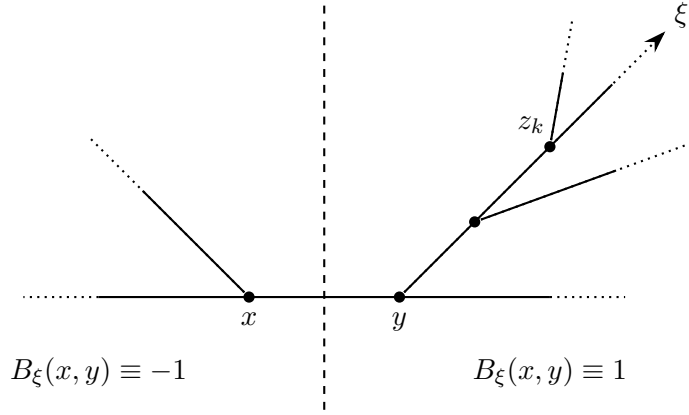}
\caption{The values of the Busemann kernel depending on $\xi$, in the case of adjacent vertices $x,y$. Pictured for $q=2$.}\label{fig:Busemann}
\end{figure}

\subsection{The representations}
Given a Banach space $W$, we make two choices in order to define the Banach space $V=L^p(\partial T, W)$. First, we choose a basepoint $x_0$ and consider the corresponding measure $\mu=\mu_{x_0}$ as above. Secondly, we choose the exponent $1\leq p < \infty$. 

With these choices, the space $V=L^p(\partial T, W)$ of Bochner $p$-summable $W$-valued function classes on $\partial T$ is a Banach space and is separable as soon as $W$ is so. We refer to the monograph~\cite{Diestel-Uhl} for very thorough background information on these Bochner spaces. A basic fact is that the step functions, namely measurable function (classes) taking only finitely many values, form a dense subspace of $V$, see e.g.\ \S II.2 in~\cite{Diestel-Uhl}. The duality theory for $V$ is more subtle and will be discussed in the proof of \Cref{ithm:more-V}.

Given an invertible continuous linear operator $\tau\colon W\to W$ of the Banach space $W$, we define a $G$-representation $\pi$ on $V$ by the ``induction'' formula
\begin{equation}\label{eq:induction}
\big(\pi(g) v\big) (\xi) = \tau^{B_\xi(x_0, g x_0)} \, v(g\inv \xi) \kern3mm \text{for} \kern3mm  g\in G, v\in V.
\end{equation}
For any given $g$, the kernel $B_\xi(x_0, g x_0)$ takes only finitely many values because the triangle inequality bounds its absolute value by $d(x_0, g x_0)$. Therefore, there is a bound on the norm of the operators $\tau^{B_\xi(x_0, g x_0)}$ occurring for the given $g$. It follows that $\pi(g) v$ remains in $L^p(\partial T, W)$ and that $\pi(g)$ defines a continuous linear operator on $V$. A similar argument shows that the representation $\pi$ is continuous in the sense that for every $v\in V$, the orbital map $g\mapsto \pi(g) v$ is norm-continuous as a map $G\to V$.

By construction, the $K$-representation on $V$ for $K=\Fix_G(x_0)$ is simply a translation representation, so that the $K$-fixed vectors $V^K$ are isometrically identified with $W$ as constant maps.

\subsection{Fixed vectors}
So far, we made no assumptions at all on the Banach space $W$ nor on the operator $\tau\colon W\to W$, except that $\tau$ is supposed invertible so that  $\tau^{B_\xi(x_0, g x_0)}$ is defined. In this generality, we can establish the following, where as always $q+1\geq 3$ is the degree of the tree.

\begin{prop}\label{prop:special-or-spherical}
Let $V_0 \se V$ be any non-zero closed $G$-invariant subspace and let $x_0, x_1$ be two adjacent vertices of $T$.

Then $V_0$ admits a non-zero vector fixed by $\Fix_G(\{x_0, x_1\})$.

Furthermore, $V_0$ admits a non-zero vector fixed by $\Fix_G(x_0)$ unless $q$ or $-q$ is an eigenvalue of $\tau$.
\end{prop}

We clarify that $\Fix_G(S)$ denotes the \emph{pointwise} fixator of a subset $S\se T$.

In the above statement, the fact that $x_0$ is the vertex that we chose for the measure $\mu=\mu_{x_0}$ is convenient, but not essential. Indeed, any automorphism sending $x_0$ to any other vertex would intertwine the entire picture.

\smallskip
We shall repeatedly use the following fact: given $v\in V$ and a compact subgroup $U<G$, the integral
\begin{equation*}
\int_U \pi(k) v \, \dd k
\end{equation*}
with respect to the normalized Haar measure on $U$ gives a $U$-fixed vector of $V$. The point is that this is a well-defined Bochner integral since the representation is orbitally continuous.

\begin{proof}[Proof of \Cref{prop:special-or-spherical}]
Recall that a subtree $S$ of $T$ is called \textbf{complete} if every vertex of $S$ is either a boundary vertex (i.e.\ adjacent to a vertex not in $S$) or has valency $q+1$ in $S$.

Examples of finite complete subtrees include the trees reduced to one vertex, or to two adjacent vertices, or the regular closed $r$-neighbourhood of any finite subtree, where $r>0$ is an integer and the metric is the combinatorial graph distance.

Since the restriction of $\pi$ to $V_0$ is a continuous representation on a (non-zero) Banach space, it is well-known that there are non-zero vectors fixed by \emph{some} compact open subgroup $U<G$ of $G$. We recall the argument. Choose any non-zero vector $v\in V_0$. Since compact open subgroups form a neighbourhood basis of the identity, we can choose such a group $U$ so that $\|\pi(k) v - v \| < \frac12 \|v\|$ for all $k\in U$. The $U$-fixed vector $\int_U \pi(k) v\,\dd k$ is in $V_0$ and it cannot be zero because the entire closed convex hull of the $U$-orbit of $v$ is in the closed ball of radius $\frac12 \|v\|$ around $v$.

We can assume $U=\Fix_G(S)$ for some finite subtree $S$ of $T$ since these groups still form a neighbourhood basis. We can furthermore assume that $S$ is complete upon replacing it by its closed $1$-neighbourhood.

Among all finite complete subtrees, we choose $S$ with the minimal number of vertices for the property that $\Fix_G(S)$ has non-zero fixed vectors in $V_0$. To prove the first statement of \Cref{prop:special-or-spherical}, it suffices to show that $S$ consist of only one or two vertices. Indeed, since $G$ acts transitively on pairs of adjacent vertices, we can then conjugate $\Fix_G(S)$ so that $S\se \{x_0, x_1\}$, or equivalently $\Fix_G(S) > \Fix_G(\{x_0, x_1\})$.

Let $D=\diam(S)$ be the diameter of $S$, namely the maximal distance $d(x,y)$ as $x,y$ range over all vertices of $S$. By the above discussion, the first statement to prove is $D\leq 1$. We suppose for a contradiction $D\geq 2$ and choose a geodesic segment of length $D$ in $S$. Conjugating $\Fix_G(S)$, we can assume that this segment starts with $x_0, x_1$; we extend it to a geodesic line $(x_k)_{k\in \ZZ}$ in $T$ so that $x_k\in S$ for all $0\leq k \leq D$.

We define a subtree $S'\se S$ of $T$ by removing from $S$ every neighbour of $x_{D-1}$ except $x_{D-2}$, recalling our assumption $D\geq 2$. By construction, $S'$ is again complete. The minimality of $S$ therefore implies that the group $\Fix_G(S')$ has no non-zero fixed vectors in $V$. In other words, we have
\begin{equation}\label{eq:average}
\int_{\Fix_G(S')} \pi(k) v \, \dd k = 0 \kern3mm \forall\, v\in V
\end{equation}
where the integral is with respect to the normalized Haar measure on $\Fix_G(S')$. We apply this to the case where $v$ is a non-zero vector fixed by $\Fix_G(S)$. Since $\Fix_G(S) < \Fix_G(x_0)$, its representation on $V$ is just the regular translation action on $L^p(\partial T, W)$. In other words, $v$ is of the form
\begin{equation*}
v = w_1 \one_{A_1} + \cdots+ w_n \one_{A_n} \kern3mm (w_i \in W)
\end{equation*}
where $\partial T = A_1 \sqcup \ldots \sqcup A_n$ is an enumeration of the $\Fix_G(S)$-orbits in $\partial T$ and at least one of the $w_i$ is non-zero.

The $\Fix_G(S)$-orbits in $\partial T$ have a simple geometric interpretation: each $A_i$ corresponds bi-univocally to a boundary vertex of $S$. Specifically, $A_i$ is the set of boundary points represented by a geodesic ray issuing from that boundary vetex and not containing any other vertex of $S$.

Consider the corresponding description of the $\Fix_G(S')$-orbits. Now $x_{D-1}$ is a boundary vertex of $S'$ but not of $S$. Moreover, $q$ of its neighbours in $S$, including $x_D$, are boundary vertices of $S$ but do not belong to $S'$. The remaining boundary vertices of $S$ and $S'$ coincide.

Therefore we can number the  $\Fix_G(S)$-orbits $A_i$ in such a way that the $\Fix_G(S')$-orbits are precisely
\begin{equation*}
(A_1 \sqcup \ldots \sqcup A_q), A_{q+1}, \ldots, A_n,
\end{equation*}
see \Cref{fig:S}.

\begin{figure}[hbt]
\centering
\begin{tikzpicture}[
    scale=0.5, 
    thick,     
    smalldot/.style = {circle, fill, inner sep=1.5pt},
    hollowdot/.style = {circle, draw, fill=white, inner sep=1.5pt} 
  ]
    
    \draw (0,0) ellipse (10cm and 8cm);
    \draw[dashed] (-1,0) ellipse (9cm and 7.6cm);
    \draw (-11,0) -- (-6,0);
    \draw (4,0) -- (11,0);
    \draw[dotted] (-6,0) -- (4,0);
    

    \node (X0) at (-9,0) [smalldot] {};  
    \node [below=2pt] at (-9,0) {$x_0$};

    \node (X1) at (-7,0) [smalldot] {};  
    \node [below=2pt] at (-6.8,0) {$x_1$};

    \node (XD-2) at (5,0) [smalldot] {};  
    \node [below=2pt] at (4.25,0) {$x_{D-2}$};

    \node (XD-1) at (7,0) [smalldot] {};  
    \node [below=2pt] at (6.5,0) {$x_{D-1}$};

    \node (XD) at (9,0) [hollowdot] {};  
    \node [below=2pt] at (9,0) {$x_D$};


    \node (L) at (-7, 4.5)  [smalldot] {}; 

    \node (RU) at (8.5, 2)  [hollowdot] {}; 

    \node (RD) at (8.5, -2)  [hollowdot] {}; 


    \draw (XD-1) -- (RU);
    \draw (XD-1) -- (RD);


    \foreach \angle in {155, 205}
        \draw (X0) -- ++(\angle:2);
      
    \foreach \angle in {25, -25}
        \draw (XD) -- ++(\angle:2);
  
    \foreach \angle in {110, 130, 150}
        \draw (L) -- ++(\angle:2);
     \draw (L) -- ++(-55:1);
     \draw[dotted] (L) -- ++(-55:2);

    \foreach \angle in {0,20,40}
        \draw (RU) -- ++(\angle:2);

    \foreach \angle in {0,-20,-40}
        \draw (RD) -- ++(\angle:2);

    \foreach \angle in {125, 235}
        \draw (X1) -- ++(\angle:1);
    \foreach \angle in {125, 235}
        \draw[dotted] (X1) -- ++(\angle:2);

    \foreach \angle in {65, -65}
        \draw (XD-2) -- ++(\angle:1);
    \foreach \angle in {65, -65}
        \draw[dotted] (XD-2) -- ++(\angle:2);


    \node at (12,0) {$A_1$};
    \node at (-12,0) {$A_n$};

    \node at (11,3) {$A_2$};
    \node at (11,-3) {$A_q$};

    \node at (-9,6.5) {$A_i$};


    \node at (9,5) {$S$};

    \node at (5,4) {$S'$};

\end{tikzpicture}
\caption{The complete trees $S$ and $S'$, and the boundary orbits $A_1, \ldots, A_n$, with $q=3$. Only one $A_i$ is pictured for $q<i<n$.}\label{fig:S}
\end{figure}
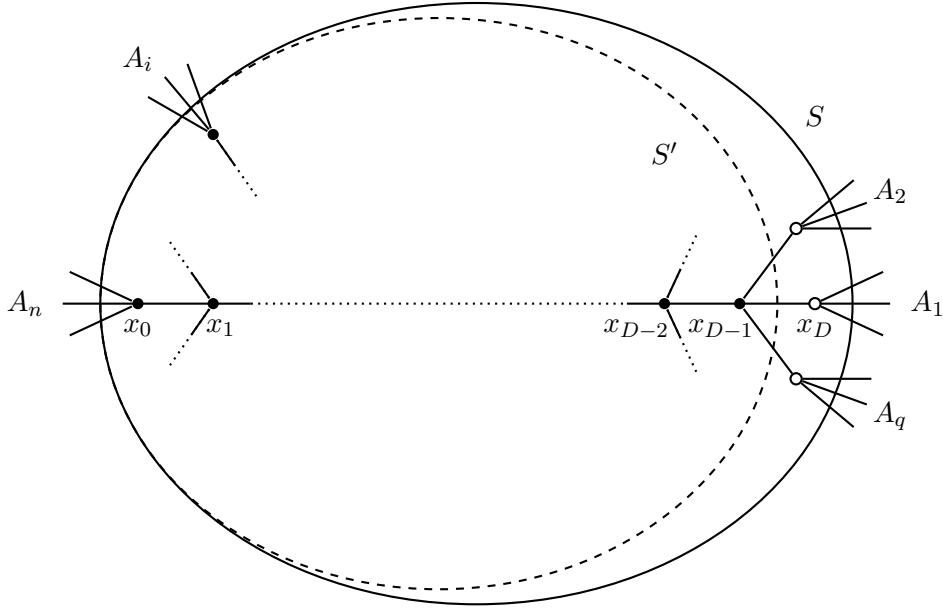

We can now compute the normalized integral~\eqref{eq:average}; 
namely, it is
\begin{equation*}
\frac{w_1 + \cdots + w_q}{q}\, \one_{(A_1 \sqcup \ldots \sqcup A_q)} +w_{q+1} \one_{A_{q+1}}  +\cdots+ w_n \one_{A_n}.
\end{equation*}
In conclusion, the vanishing of that integral amounts to $w_1 + \cdots + w_q=0$ and $w_i=0$ when $i>q$.

We now choose an element $g\in G$ such that $g x_k = x_{k-1}$ for all $k$. Recall that $x_D$ is one of the boundary vertices of $S$ corresponding to one of the first $q$ orbits $A_1, \ldots , A_q$; say it is $A_1$. In particular,
\begin{equation*}
g A_1 = A_1 \sqcup \ldots \sqcup A_q.
\end{equation*}
This time we apply the condition~\eqref{eq:average} to $\pi(g) v$ instead of $v$. Using~\eqref{eq:induction}, this gives for every $\xi\in\partial T$ the condition
\begin{equation}\label{eq:average-xi}
0 = \int_{\Fix_G(S')} (\pi(kg) v)(\xi) \, \dd k = \int_{\Fix_G(S')}  \tau^{B_\xi(x_0, k g x_0)} \,  v(g\inv k\inv\xi) \, \dd k.
\end{equation}
By the choice of $g$, we have $g\inv k\inv\xi\in A_1$ for all $\xi\in A_1  \sqcup \ldots \sqcup A_q$. For these $\xi$, we further have
\begin{equation*}
B_\xi(x_0, k g x_0) = B_{k\inv \xi}(x_0,  g x_0) = B_{k\inv \xi}(x_0,  x_1) =-1
\end{equation*}
by definition of the Busemann cocycle~\eqref{eq:Busemann}. Therefore~\eqref{eq:average-xi} implies $\tau\inv w_1 = 0$ and therefore $w_1=0$.

Finally, we exploit the symmetry among all $q$ vertices of $S\smallsetminus S'$ as follows. We can choose $g$ to translate a geodesic line passing through $x_0$ and through any of those vertices instead of $x_D$; then the exact same argument shows $w_i=0$ for any $1\leq i \leq q$ instead of $i=1$. This completes the proof that $v=0$ and this contradiction shows that $D=\diam(S)\geq 2$ is impossible.

\medskip
We now turn to the second statement of \Cref{prop:special-or-spherical} and write  $K=\Fix_G(x_0)$. Suppose that $V_0$ has no non-zero $K$-fixed vector. Then every $v\in V_0$ satisfies $\int_K \pi(k) v \,\dd k=0$, which can be written as
\begin{equation}\label{eq:average-K}
\int_{\partial T} v(\xi) \,\dd \mu(\xi)=0.
\end{equation}
We know from the first statement that $V_0$ contains a non-zero element $v= w_1 \one_{A_1} + w_2 \one_{A_2}$, where $A_1$ denotes the set of rays starting with $(x_0, x_1)$ whereas $A_2$ denotes those starting with $(x_1, x_0)$. Alternatively, this can be formulated as
\begin{equation}\label{eq:A_i}
\xi\in A_1 \Longleftrightarrow B_\xi(x_0, x_1) = 1,\kern3mm \xi\in A_2 \Longleftrightarrow B_\xi(x_0, x_1) = -1.
\end{equation}
Since $\mu(A_2) = q \mu(A_1)$, condition~\eqref{eq:average-K} reads $w_1 + q w_2 = 0$. In particular, $w_2\neq 0$. Consider now an element $h\in G$ exchanging $x_0$ and $x_1$. Using~\eqref{eq:A_i} in the definition of $\pi$ gives
\begin{equation}\label{eq:h-acts}
\pi(h) v =  \tau(w_2) \,\one_{A_1} + \tau\inv(w_1) \, \one_{A_2}.
\end{equation}
Applying again~\eqref{eq:average-K} gives $\tau(w_2) + q \tau\inv(w_1) = 0$ and we obtain $\tau^2(w_2) = q^2 w_2$. Since $w_2\neq 0$, we conclude that $W$ admits $q$ or $-q$ as eigenvalue.
\end{proof}

\begin{rem}\label{rem:non-complete}
There is only one point in the above proof where we used that $V_0$ is closed in a Banach space, namely to argue by integration that $V_0$ contains a non-zero vector fixed by some compact open subgroup $U<G$. All other integrals in the proof where of the following form: given two compact open subgroups $U,U'<G$, we integrate $\int_{U'} \pi(k) v \,\dd k$ where $v$ is $U$-fixed. Thus this integral is just a finite linear combination indexed by the finite set $U'/(U'\cap U)$, and this makes sense in any vector subspace of $V$.

In conclusion, the entire statement of \Cref{prop:special-or-spherical} holds true as soon as $V_0\se V$ is a $G$-invariant vector subspace containing a non-zero smooth vector.
\end{rem}

\subsection{Invariant subspaces}
We now specialize the above construction to invertible operators $\tau$ obtained as follows. Choose a holomorphic branch $\sqrt{\cdot}$ of the square root on the (unusual) cut domain $\CC \smallsetminus \RR_{\geq 0}$ and consider 
\begin{equation*}
\fhi(z) = \frac{z + \sqrt{z^2 - 4 q}}2 \kern6mm \text{(where defined)}.
\end{equation*}
In particular, $\fhi$ is defined and holomorphic on the open disc $\DD_{2q^{1/2}}$ of radius $2q^{1/2}$ around zero; see \Cref{fig:domain}. Here we write $q^{1/2}$ for the (usual) real square root.

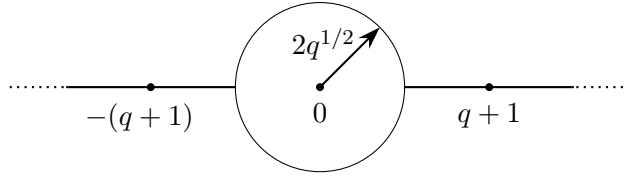
\begin{figure}[htb]
\centering
\begin{tikzpicture}[
    scale=0.75, 
    thick 
  ]
    
    \def\R{1.5} 
    
    \draw[thin] (0,0) circle (\R);
    
    \fill (0,0) circle (2pt);
    \node [below=2pt] at (0,0) {$0$};
    
    \draw [ -{Stealth[length=3mm, width=2mm]} ] (0,0) -- (45:\R) 
              node[midway, xshift=-10pt, yshift=5pt] {$2q^{1/2}$};

    \draw (- \R, 0) -- (-\R - 3, 0);
    \draw[dotted] (- \R - 4, 0) -- (-\R - 3, 0);

    \fill (\R + 1.5, 0) circle (2pt);
    \node [below=2pt] at (\R +1.5,0) {$q+1$};
   
    \fill (-\R - 1.5, 0) circle (2pt);
    \node [below=2pt] at (-\R -1.7,0) {$-(q+1)$};
   
    \draw (\R, 0) -- (\R + 3, 0);
    \draw[dotted] (\R + 4, 0) -- (\R + 3, 0);


\end{tikzpicture}
\caption{The domain of $\fhi$ has two horizontal cuts; it contains $\DD_{2q^{1/2}}$ but not $\pm (q+1)$.}\label{fig:domain}
\end{figure}

Given any continuous linear operator $\alpha\colon W\to W$ of norm $\|\alpha\|< 2q^{1/2}$, we therefore obtain by holomorphic functional calculus an operator $\tau=\fhi(\alpha)$. Moreover $\tau$ is invertible; explicitly, we have $\tau\inv =\fhi(\alpha)\inv$, where $\fhi(z)\inv= \frac1{2q} (z - \sqrt{z^2 - 4 q})$. By construction, we have
\begin{equation*}
\tau + q \tau\inv = \alpha
\end{equation*}
since $\fhi(z) + q \fhi(z)\inv=z$.

\medskip
The main motivation for the construction of the $G$-representation $\pi$ on $V=L^p(\partial T, W)$ is that we can control $G$-invariant subspaces of $V$ in terms of $\alpha$-invariant subspaces of $W$, as follows.

\begin{prop}\label{prop:induct-dico}
Let $V_0 \se V$ be any non-zero closed $G$-invariant subspace and let $K=\Fix_G(x_0)$.

Then the image $W_0\se W$ of $V_0^K$ under the identification $V^K \cong W$ is a non-zero $\alpha$-invariant closed subspace.
\end{prop}

\begin{proof}
 The spectral mapping theorem shows that $\pm q$ is not an eigenvalue of $\tau$, because $\pm q$ lie outside $\fhi(\DD_{2q^{1/2}})$ since $\pm(q+1)$ are not in $\DD_{2q^{1/2}}$. Therefore \Cref{prop:special-or-spherical} shows that $W_0$ is non-zero. We need to show $\alpha(w)\in W_0$ for all $w\in W_0$.

Consider $v=w \one_{\partial T}$ in $V_0^K$ and an element $h\in G$ exchanging $x_0$ and $x_1$ as in the end of the Proof of \Cref{prop:special-or-spherical}. The computation for~\eqref{eq:h-acts} gives
\begin{equation*}
\pi(h) v =  \tau(w) \,\one_{A_1} + \tau\inv(w) \, \one_{A_2}.
\end{equation*}
On the other hand, $\int_K \pi(k) \pi(h) v \,\dd k$ is again in $V_0^K$, so that it is of the form $w' \one_{\partial T}$ for some $w'\in W_0$. Computing this integral as $\int_{\partial T}  (\pi(h) v)(\xi) \,\dd \mu(\xi)$ and using $\mu(A_1) = 1/(q+1)$, $\mu(A_2)=q/(q+1)$, we deduce
\begin{equation*}
w' =\frac1{q+1} \tau(w) + \frac{q}{q+1} \tau\inv(w) = \frac1{q+1} \alpha (w).
\end{equation*}
This confirms $\alpha(w)\in W_0$.
\end{proof}

We have now collected enough properties of $V$ to deduce that it is irreducible if $\alpha$ has no invariant subspace.

\begin{thm}\label{thm:irred}
Let $W$ be a Banach space admitting a linear operator $\alpha$ without invariant closed proper subspace and of norm  $\|\alpha\|< 2q^{1/2}$.

Then the $G$-representation on $V=L^p(\partial T, W)$ constructed above is topologically irreducible for every $1\leq p < \infty$.
\end{thm}

\begin{proof}[Proof of \Cref{thm:irred}]
Let $V_0 \se V$ be a non-zero closed $G$-invariant subspace. By density of the step functions, it suffices to prove that $V_0$ contains all vectors of the form $w\one_C$ where $w\in W$ and $C\se \partial T$ ranges over a basis of the topology of $\partial T$. Such a basis is provided by the collection of all ``half-trees''
\begin{equation*}
C = \Big\{ \xi\in \partial T : B_\xi(x, y) = 1 \Big\} \kern3mm \text{where}\kern3mm x,y\in T \text{ with } d(x,y)=1.
\end{equation*}
Note that the other ``half'' $C' = \partial T \smallsetminus C$ corresponds to the only other value $ B_\xi(x, y) = -1$; see \Cref{fig:Busemann}. By $G$-invariance of $V_0$, it suffices to establish this in the special case where $x=x_0$.

Choose $g\in G$ such that $g x_0 = y$. Given $w'\in W$, the definition~\eqref{eq:induction} gives
\begin{equation*}
\pi(g) \left(w'\, \one_{\partial T}\right) = \tau(w')\, \one_C + \tau\inv(w') \,\one_{C'}
\end{equation*}
and therefore
\begin{equation}\label{eq:half-half}
\pi(g) \left(w'\, \one_{\partial T} \right) - \tau\inv(w')\, \one_{\partial T} = \big(\tau(w') - \tau\inv(w') \big)\,\one_C.
\end{equation}
\Cref{prop:induct-dico} implies that $V_0$ contains all elements of the form  $w''\, \one_{\partial T}$ with $w''\in W$. In particular, $V_0$ contains both $w'\, \one_{\partial T}$ and $\tau\inv(w')\, \one_{\partial T}$. By $G$-invariance it also contains $\pi(g) \left(w'\, \one_{\partial T} \right)$ and thus we conclude from~\eqref{eq:half-half} that $V_0$ contains, for any $w'\in W$, the element $\big(\tau(w') - \tau\inv(w') \big)\,\one_C$.

Therefore, it only remains to show that any $w\in W$ is of the form $\tau(w') - \tau\inv(w')$; that is, to show that $\tau - \tau\inv$ is onto. Actually that map is even invertible: this follows from the spectral mapping theorem because the function $\fhi-\fhi\inv$ does not vanish on $\DD_{2q^{1/2}}$. In fact it never vanishes because its domain does not contain $\pm (q+1)$.
\end{proof}

At this point, \Cref{ithm:inad} follows from the negative solution to the invariant subspace problem.

\begin{proof}[Proof of \Cref{ithm:inad}]
Enflo~\cite{Enflo76,Enflo87} and Read~\cite{Read84} have proved that there exists an infinite-dimensional Banach space $W$ and a continuous linear operator $\alpha\colon W\to W$ without any invariant closed proper subspace. If needed, we rescale $\alpha$ to ensure $\|\alpha\|< 2q^{1/2}$.

Then the $G$-representation $\pi$ on $V=L^p(\partial T, W)$ is topologically irreducible by \Cref{thm:irred} and inadmissible by \Cref{prop:induct-dico}.

\smallskip
For the additional statement, consider the subspace $\smooth{V}$ of smooth vectors and suppose that it contains a non-zero algebraically irreducible subrepresentation $V_1 \se \smooth{V}$. In view of \Cref{rem:non-complete}, we can still apply \Cref{prop:special-or-spherical} and deduce $V_1^K\neq 0$. On the other hand, Olshanskii proved (Theorem~1 in~\cite{Olshanskii75}) that algebraically irreducible $G$-representations are admissible, so that $V_1^K$ is finite-dimensional. (In the present case, this can also be deduced from the fact that $(G,K)$ is a Gelfand pair, using that $V_1$ must have countable algebraic dimension. In fact this argument amounts to a special case of Olshanskii's Lemma~3 in~\cite{Olshanskii75}.) Now the argument of \Cref{prop:induct-dico} shows that $\alpha$ preserves a finite-dimensional subspace of $W$, which is absurd.
\end{proof}

We can obtain more Banach-theoretic information on $V$ from information on $W$ provided by the extensive work of Read.

\begin{proof}[Proof of \Cref{ithm:more-V}]
The quasi-reflexive Banach space $W$ discovered by James~\cite{James51} is isometrically isomorphic to its bidual $W^{**}$. Read proves in~\cite{Read89} that this space admits continuous linear operator without invariant closed proper subspace.  We apply \Cref{ithm:inad} to this situation, with the additional restriction $p>1$. We only need to justify that $V=L^p(\partial T, W)$ is also isometrically isomorphic to its own bidual (but we point out that it is not quasi-reflexive).

Write $p'$ for the H{\"o}lder conjugate $p'=p/(p-1)$. Integration over $\partial T$ yields a canonical isometric embedding of $L^{p'}(\partial T, W^*)$ into $V^*$, where~$^*$ denotes the Banach dual. This embedding is an isometric isomorphism if and only if $W^*$ has the Radon--Nikod{\'y}m property, see Theorem~1 of \S IV.1 in~\cite{Diestel-Uhl}. On the other hand, every separable dual enjoys the Radon--Nikod{\'y}m property, see Theorem~1 of \S III.3 in~\cite{Diestel-Uhl}. The dual $W^*$ is separable; this holds for any space whose dual is separable: see Banach, Th{\'e}or{\`e}me~12 of \S XI.9 in~\cite{Banach_book}. Here, $W^{**}$ is separable because $W$ is separable from its very definition. 

In summary, $V^*$ can be identified isometrically with $L^{p'}(\partial T, W^*)$. Since $p'<\infty$ and $p''=p$, we can repeat exactly the same Radon--Nikod{\'y}m arguments to identify $V^{**}$ with $L^{p}(\partial T, W^{**})$. By James's theorem, this implies that $V^{**}$ is isometrically isomorphic to $V$.

\smallskip
Concerning $V=L^1$, we obtain it from another of Read's results. He proved in~\cite{Read85} that $W=\ell^1(\NN)$ also admits a continuous linear operator without invariant closed proper subspace. This time, we apply \Cref{ithm:inad} with $p=1$; Fubini's theorem implies that $V=L^1(\partial T, \ell^1(\NN))$ can be canonically identified with $L^1(\partial T \times \NN)$. Since $\partial T \times \NN$ is a countably generated non-atomic measure space, we obtain indeed the standard Lebesgue space $L^1$. 
\end{proof}

Of course with $p=2$ we obtain a Hilbert space $V=L^2(\partial T, W)$ if $W$ itself is a Hilbert space; the original invariant subspace problem concerns that case, but remains notoriously unsolved.


\bibliographystyle{./amsalpha-nobysame}
\bibliography{../BIB/ma_bib}

\end{document}